\documentclass[a4paper,10pt]{amsart}
\usepackage{amsmath,amssymb,amsthm}
\newcommand{\CM}{\overline{\operatorname{\mathcal{M}}}}

\newcommand{\IM}{\operatorname{\mathcal{M}}}
\newcommand{\BB}{\operatorname{\mathcal{B}}}
\newcommand{\EE}{\operatorname{\mathcal{E}}}
\newcommand{\LL}{\operatorname{\mathcal{L}}}

\newcommand{\PP}{\operatorname{\mathfrak{P}}}

\newcommand{\Tor}{\operatorname{Tor}}

\newcommand{\CP}{\mathbb{C}\mathbb{P}^1}

\newcommand{\CR}{\bar{\partial}}

\newcommand{\Aut}{\operatorname{Aut}}

\newcommand{\SFT}{\operatorname{SFT}}

\newcommand{\ev}{\operatorname{ev}}

\newcommand{\Crit}{\operatorname{Crit}}

\newcommand{\cst}{\operatorname{const}}

\newcommand{\CZ}{\operatorname{CZ}}

\newcommand{\ind}{\operatorname{ind}}
\newcommand{\coker}{\operatorname{coker}}
\newcommand{\del}{\partial}

\newcommand{\RS}{\IR \times S^1}

\newcommand{\IC}{\operatorname{\mathbb{C}}}
\newcommand{\IZ}{\operatorname{\mathbb{Z}}}

\newcommand{\IR}{\operatorname{\mathbb{R}}}
\newcommand{\IN}{\operatorname{\mathbb{N}}}

\newcommand{\IG}{\operatorname{\mathbf{G}}}

\newcommand{\Ih}{\operatorname{\mathbf{h}}}
\newcommand{\Ig}{\operatorname{\mathbf{g}}}

\newcommand{\If}{\operatorname{\mathbf{f}}}
\newcommand{\Coker}{\operatorname{Coker}}

\newtheorem{theorem}{Theorem}[section]
\newtheorem{proposition}[theorem]{Proposition}
\newtheorem{definition}[theorem]{Definition}

\newtheorem{corollary}[theorem]{Corollary}

\title{Local symplectic field theory}
\author{Oliver Fabert}
\begin{document}

\maketitle

\begin{abstract}
Generalizing local Gromov-Witten theory, in this paper we define a local version of symplectic field theory. When the symplectic manifold with cylindrical ends is four-dimensional and the underlying simple curve is regular by automatic transversality, we establish a transversality result for all its multiple covers and discuss the resulting algebraic structures. 
\end{abstract}

\tableofcontents

\markboth{O. Fabert}{Local SFT} 

\section*{Introduction}
In this paper we define a local version of Eliashberg-Givental-Hofer's symplectic field theory (SFT), see \cite{EGH}. It provides a topological quantum theory approach to local Gromov-Witten theory in the same way as standard SFT provides a topological quantum field theory approach to standard Gromov-Witten theory. While in local Gromov-Witten theory one counts multiple covers over a fixed closed holomorphic curve, see \cite{LP}, \cite{BP}, in local SFT we count multiple covers over punctured holomorphic curves. \\

Instead of getting invariants for contact manifolds, we now get the invariants for closed Reeb orbits that were already studied in \cite{F2} and \cite{F3}. Note that for the orbit curves we used an infinitesimal energy estimate to show that multiple covers 
of orbit cylinders are isolated in the moduli space of holomorphic curves. In this paper we show that the dimension bounds on the kernel of the linearized Cauchy-Riemann operator established in \cite{Wen} (using positivity of intersections in dimension four) can be used to obtain the required isolatedness result for rational multiple covers when the underlying simple (rational) curve is sufficiently nice.

\begin{theorem} 
Assume that the rigid holomorphic curve $v:S\to X$ is immersed and that all asymptotic orbits are Morse nondegenerate and elliptic. If $\ind u=0$ for rational multiple covers $u=v\circ\varphi$ in $\IM_{v,d}(\Gamma^+,\Gamma^-)$, then every infinitesimal deformation of $u$ as a holomorphic curve is again a multiple cover of $v$. Furthermore the cokernels of the linearized Cauchy-Riemann operator $\CR_J$ fit together to a smooth \emph{obstruction bundle} $\overline{\Coker}_v\CR_J=\overline{\Coker}_{v,d}(\Gamma^+,\Gamma^-)$ over the compactified moduli space $\CM_v=\CM_{v,d}(\Gamma^+,\Gamma^-)$. 
\end{theorem}

Using these obstruction bundles we can solve the transversality problem for multiply-covers of immersed curves with elliptic orbits without employing the polyfold machinery from \cite{HWZ}, see (\cite{MDSa}, section 7.2) for the general approach. 

\begin{proposition} Let $\nu$ be a section in the cokernel bundle $\Coker\CR_J\subset\EE|_{\IM}$ over the moduli space $\IM=\CR_J^{-1}(0)\subset\BB$, which is extended (using parallel transport and cut-off functions, as described in \cite{F2},\cite{MDSa},\cite{LP}) to a section in the full Banach space bundle $\EE\to\BB$. Then it holds:
\begin{itemize} 
\item The perturbed moduli space $\IM^{\nu}=(\CR_J+\nu)^{-1}(0)$ agrees with the zero set of $\nu$ in $\IM$,   
\begin{equation*} \IM^{\nu}=\nu^{-1}(0). \end{equation*}
\item If $\nu$ is a transversal section in $\Coker\CR_J$, then $\CR^{\nu}_J$ is a transversal section in $\EE$, i.e., $\IM^{\nu}$ is regular.
\item The linearization of $\nu$ at every zero is a compact operator, so that the linearizations of $\CR_J$ and $\CR_J^{\nu}$ belong to the same class of Fredholm operators.
\end{itemize}
\end{proposition}

We then show how this can be used to define morphisms in the local version of Eliashberg-Givental-Hofer's symplectic field theory introduced by the author in \cite{F2}. More precisely, we show that immersed holomorphic curves with elliptic asymptotic orbits in four-dimensional symplectic cobordisms define morphisms between the local SFT invariants assigned to their asymptotic closed Reeb orbits. While in standard SFT one collects the information about all moduli spaces of holomorphic curves in $X$ by defining a potential $\If$, we now define a local SFT potential $\If_v\in\LL_{\Gamma'^+,\Gamma'^-}$ counting only multiple covers of the fixed rigid immersed curve with elliptic orbits $v:S\to X$.

\begin{definition} For every choice of obstruction bundle sections $(\bar{\nu})$ coherently connecting the coherent collections of obstruction bundle sections $(\bar{\nu}^{\pm})$ chosen for all positive and negative asymptotic Reeb orbits $\gamma^{\pm}\in\Gamma'^{\pm}$ of $v$, we define the local SFT potential of a rigid immersed holomorphic curve $v$ with elliptic orbits by $$\If_v\stackrel{!}{=}\If_v^{(\bar{\nu})}=\sum_{\Gamma^+,\Gamma^-} \frac{1}{s^+!s^-!\kappa_{\Gamma^+}\kappa_{\Gamma^-}}\#\CM^{\bar{\nu}}_{v,d}(\Gamma^+,\Gamma^-) q_-^{\Gamma^-} p_+^{\Gamma^+},$$ where $\CM^{\bar{\nu}}_{v,d}(\Gamma^+,\Gamma^-)=\bar{\nu}^{-1}(0)\subset\CM_{v,d}(\Gamma^+,\Gamma^-)$. \end{definition} 

Furthermore we will use the result in \cite{F2} to discuss how the algebraic count of multiple covers of a immersed curve with elliptic orbits depends on all the auxiliary choices. Here we prove the following 

\begin{theorem} Assume that the coherent collections of sections $(\bar{\nu}_{\pm})$ are fixed for all asymptotic Reeb orbits $\gamma^{\pm}\in\Gamma'^{\pm}$ of $v$. Then the local SFT potential $\If_v=\If_v^{(\bar{\nu})}$ of $v$ is independent of the chosen collection of sections $(\bar{\nu})$ coherently connecting $(\bar{\nu}^+)$ and $(\bar{\nu}^-)$. In particular, the algebraic count of multiple covers of the immersed curve with elliptic orbits $v$ is well-defined. \end{theorem}

At the end we illustrate how the new local SFT invariants can be used to obtain richer obstructions against stable embeddings of hypersurfaces in four-dimensional symplectic manifolds. After introducing additional marked points and gravitational descendants (translated into branching conditions as in \cite{F3}), we reprove the following result due to Welschinger.

\begin{theorem} Assume that a closed oriented Lagrangian surface $L$ in a closed symplectic four-manifold has a homologically nontrivial intersection with an exceptional sphere $\Sigma$. Then $L$ must be diffeomorphic to $S^2$ or $S^1\times S^1$. \end{theorem}

This paper is organized as follows: After recalling the basic definitions and results of symplectic field theory in subsection 1.1, we discuss the ideas and main definitions of its local 
version in subsection 1.2. While in subsection 1.3 we use the results in \cite{Wen} to prove the desired isolatedness result for multiple covers which also proves the existence 
of a smooth finite-dimensional obstruction bundle over the corresponding moduli spaces, we show in 1.4 how the transversality problem for the Cauchy-Riemann operator can be solved 
by choosing transversal and coherent sections in these bundles and discuss in 1.5 how the resulting algebraic count of multiple covers depends on these auxiliary choices. In section 
two we then show how local SFT methods can be applied to stable embedding problems of hypersurfaces in symplectic blow-ups. After introducing gravitational descendants via branching conditions in 2.1, we explicitly compute a contribution to the local Gromov-Witten descendant potential of an exceptional sphere using topological recursion in 2.2. In subsection 2.3 we then use our computation to prove equations for the local SFT potential and finally prove our obstruction to stable embeddings of hypersurfaces in 2.4. \\

The author is deeply indebted to Chris Wendl for explaining to him his work on automatic transversality in long discussions, which were the starting point for this project. Furthermore he wants to thank Kai Cieliebak for his ideas concerning the generalization of the result in \cite{F3} and also thanks E. Ionel, T. Parker and C. Taubes for interesting discussions on this topic during his stay at the MSRI in Berkeley. This paper was written when the author was a research assistant at the University of Augsburg. He is grateful for the financial support and the great working environment and also thanks Prof. K. Wendland's ERC Starting Independent Researcher Grant (StG No. 204757-TQFT) for further support.

\section{Local symplectic field theory}

\subsection{Symplectic field theory}
Symplectic field theory was defined by Eliashberg, Givental and Hofer in their paper \cite{EGH} and is designed to describe in a unified way the theory of pseudoholomorphic curves in symplectic and contact topology. In particular, it defines a functor from a geometric category to an algebraic category. The objects of the geometric category are contact manifolds (more generally, manifolds with stable Hamiltonian structure) $V$ of dimension $2n-1$ ($n\geq 1$) , while the morphisms from one contact manifold $V^-$ to another contact manifold $V^+$ are strong symplectic cobordisms $X$ of dimension $2n$ from $V^-$ to $V^+$, that is, strong symplectic fillings of the disconnected union $-V^-\cup V^+$. \\

The functor SFT defines invariants for contact manifolds $V$, denoted by $\SFT(V)$,  by counting $J$-holomorphic curves in cylindrical manifolds $\IR\times V$ equipped with a compatible almost complex structure $J$, which is cylindrical in the sense that it is $\IR$-invariant, preserves the contact distribution, $\xi=\ker\lambda=TV\cap JTV$, and maps the $\IR$-direction to the Reeb vector field $R\in\ker d\lambda$, $\lambda(R)=1$. For this recall that a contact one-form $\lambda$ defines a vector field $R$ on $V$ by 
$R\in\ker d\lambda$ and $\lambda(R)=1$, which is called the Reeb vector field. \emph{Throughout the paper we assume that 
the contact form is Morse in the sense that all closed orbits of the 
Reeb vector field are (Morse) nondegenerate in the sense that one is not an eigenvalue of the linearized return map; in particular, the set 
of closed Reeb orbits is discrete.} The invariants are defined by counting
$J$-holomorphic curves $u$ in $\IR\times V$. Let $\Gamma^+,\Gamma^-$ be two ordered sets of closed (unparametrized) orbits $\gamma$ of the Reeb vector field $R$ on $V$. Note further that in this paper \emph{we just restrict to the case of rational holomorphic curves.} Then the (parametrized) moduli space $\IM^0_{V,A}(\Gamma^+,\Gamma^-)$ consists of tuples $(u,j)$, where $j$ is a complex structure on the sphere $S=S^2-\{z^{\pm}_1,...,z^{\pm}_{s^{\pm}}\}$ with $s=s^++s^-$ punctures ($s^{\pm}=\#\Gamma^{\pm}$) removed and maps $u:(S,j) \to(\IR\times V,J)$  satisfying the Cauchy-Riemann equation 
\begin{equation*}
 \CR_J u = du + J(u) \cdot du \cdot j = 0.
\end{equation*}
Assuming we have chosen cylindrical holomorphic coordinates $\psi^{\pm}_k: \IR^{\pm} 
\times S^1 \to (S,j)$ around each puncture $z^{\pm}_k$ in the sense that $\psi_k^{\pm}(\pm\infty,t)=z_k^{\pm}$, 
the map $u$ is additionally required to show for all $k=1,...,n^{\pm}$ the 
asymptotic behaviour
\begin{equation*}
 \lim_{s\to\pm\infty} (u \circ \psi^{\pm}_k) (s,t+t_0) = 
 (\pm \infty,\gamma^{\pm}_k(T^{\pm}_kt))
\end{equation*}
with some $t_0\in S^1$ and the orbits $\gamma^{\pm}_k\in \Gamma^{\pm}$, where $T^{\pm}_k>0$ denotes period of $\gamma^{\pm}_k$. In particular, note that in the asymptotic condition is independent of the parametrization of the closed Reeb orbit. In order to assign an \emph{absolute} homology class $A\in H_2(V)$ to a holomorphic curve $u:(S,j)\to(\IR\times V,J)$ we have to 
employ spanning surfaces $u_{\gamma}$ connecting a given closed Reeb orbit $\gamma$ in $V$ to a linear combination of 
circles $c_s$ representing a basis of $H_1(V)$, $\del u_{\gamma} = \gamma - \sum_s n_s\cdot c_s$ in order to define 
$A = [u_{\Gamma^+}] + [u(S)] - [u_{\Gamma^-}]$, where $[u_{\Gamma^{\pm}}] = \sum_{n=1}^{s^{\pm}} [u_{\gamma^{\pm}_n}]$ viewed as singular chains. \\

Observe that when the number of punctures is less than three the corresponding subgroup $\Aut(S,j)$ with $(S,j)=\RS, \IC, \CP$ of the group of Moebius transformations acts on elements in $\IM^0_{V,A}(\Gamma^+,\Gamma^-)$ in an obvious way, 
\begin{equation*} \varphi.(u,j) = (u\circ\varphi^{-1},j),\;\;\;\varphi\in\Aut(S,j), \end{equation*}
and we obtain the moduli space $\IM=\IM_{V,A}(\Gamma^+,\Gamma^-)$ studied in symplectic field theory by dividing out this action and the natural $\IR$-action on the target manifold $(\IR\times V,J)$. \\    

To every strong symplectic cobordism $X$ from $V^-$ to $V^+$, the functor SFT assignes morphisms $\SFT(X)$ from the invariant $\SFT(V^-)$ to the invariant $\SFT(V^+)$ by counting $J$-holomorphic curves in $X$, where the $\omega$-compatible almost complex structure $J$ agrees with $J^{\pm}$ on the cylindrical ends $\IR^{\pm}\times V^{\pm}$ of $X$. For the latter observe that here and in what follows we do not distinguish between the strong symplectic filling and its (non-compact) completion. Indeed, let $X=(X,\omega)$ be a symplectic manifold with cylindrical ends $(\IR^+\times V^+,\lambda^+)$ and 
$(\IR^-\times V^-,\lambda^-)$ in the sense of (\cite{BEHWZ}, section 3) which is equipped with an almost complex structure 
$J$ which agrees with the cylindrical almost complex structures $J^{\pm}$ on $\IR^+\times V^+$. Then we 
study $J$-holomorphic curves $u:(S,j)\to(X,J)$ which are asymptotically cylindrical over 
chosen collections of orbits $\Gamma^{\pm}=\{\gamma^{\pm}_1,...,
\gamma^{\pm}_{n^{\pm}}\}$ of the Reeb vector fields $R^{\pm}$ in $V^{\pm}$ as the $\IR^{\pm}$-factor tends 
to $\pm\infty$, see \cite{BEHWZ}. We now denote by $\IM^0_{X,A}(\Gamma^+,\Gamma^-)$, $\IM_{X,A}(\Gamma^+,\Gamma^-)$ and $\CM_{X,A}(\Gamma^+,\Gamma^-)$ the corresponding moduli spaces of rational curves in $X$, where it is important to note that for passing from $\IM^0_{X,A}(\Gamma^+,\Gamma^-)$ to $\IM_{X,A}(\Gamma^+,\Gamma^-)$ we do no longer divide out a symmetry on target anymore. After choosing abstract perturbations using polyfolds as described above, we again find that $\CM_{X,A}(\Gamma^+,\Gamma^-)$ is a weighted branched manifold with boundaries and corners of dimension equal to the Fredholm index of the Cauchy-Riemann operator for $J$. \\

Furthermore it was shown in (\cite{BEHWZ}, theorem 10.1 and 10.2) that $\IM_{V,A}(\Gamma^+,\Gamma^-)$ and $\IM_{X,A}(\Gamma^+,\Gamma^-)$ can be compactified to moduli spaces $\CM_{V,A}=\CM_{V,A}(\Gamma^+,\Gamma^-)$ and $\CM_{X,A}=\CM_{X,A}(\Gamma^+,\Gamma^-)$ by adding moduli space of multi-floor curves with nodes. After choosing abstract perturbations using polyfolds (see \cite{HWZ}) we get that $\CM_{V,A}$ and $\CM_{X,A}$ is a branched-labelled manifold with boundaries and corners of dimension 
equal to the Fredholm index of the Cauchy-Riemann operator for $J$ (minus one in the first case, where the latter accounts for dividing out the one-dimensional $\IR$-action on the target). In particular, the moduli spaces $\CM_{V,A}$ and $\CM_{X,A}$ have codimension-one boundary given by (fibre) products $\CM_{V,A_1}\times\CM_{V,A_2}$ and $\CM_{X,A_1}\times\CM_{V^-,A_2}\cup\CM_{V^+,A_1}\times\CM_{X,A_2}$ ($A_1+A_2=A$) of lower-dimensional moduli spaces, respectively. \\

\subsection{A local version of symplectic field theory}

Note that for $n=1$ the one-dimensional contact manifold $V$ consists of a copies of circles, while a two-dimensional symplectic cobordism from $V^-$ to  $V^+$ is nothing else but a Riemann surface $S$ with $s^-$ negative and $s^+$ positive punctures, i.e., points removed, where $s^{\pm}$ denotes the number of components of $V^{\pm}$. While the SFT-invariants for $V=S^1$ count branched coverings of the cylinder $\RS$, the morphism $\SFT(S)$ is defined by counting branched coverings of $S$.  \\ 

While for $n=1$ the SFT functor is easily understood, researchers were looking for computable examples, which can be viewed as an intermediate step between the case of Riemann surfaces and the case of general symplectic manifolds. And indeed, it is well-known in Gromov-Witten theory that one can define a \emph{local} version  of it, see \cite{BP} and \cite{LP}, by counting multiple covers of a fixed (simple) rigid $J$-holomorphic curve $v: (S^2,i)\to (X,J)$. Indeed it can be shown that under certain assumptions on $v$ the submoduli spaces $\CM_{v,d}$ of $d$-fold branched coverings $u=v\circ\varphi: (S^2,i)\to(S^2,i)\to(X,J)$ is a connected component of the moduli space $\CM_{X,d[v]}$ of general holomorphic maps to $X$. \\

In this paper we define a local version of symplectic field theory, which provides a TQFT approach to local Gromov-Witten theory in the same way as 
standard symplectic field theory provides a TQFT approach to standard Gromov-Witten theory. \\

The corresponding invariants $\SFT(\gamma)$ for closed Reeb orbits $\gamma$ were introduced by the author in \cite{F2}, \cite{F3} by counting holomorphic curves in the moduli spaces $\CM_{\gamma,d}(\Gamma^+,\Gamma^-)\subset\CM_{V,0}(\Gamma^+,\Gamma^-)$ of branched covers $\varphi: (S,j)\to(\IR\times\gamma,J)$. In the present paper we show that immersed holomorphic curves with elliptic orbits $v$ can be used to define morphisms $\SFT(v)$ between the invariants $\SFT(\Gamma^+)$ and $\SFT(\Gamma^-)$ assigned to its asymptotic orbits. For this we again show that the submoduli spaces $$\CM_{v,d}(\Gamma^+,\Gamma^-)\subset\CM_{X,d[v]}(\Gamma^+,\Gamma^-)$$ of multiple covers $u=v\circ\varphi: (S,j)\to(S',j')\to(X,J)$ are isolated in the space of all holomorphic maps in the sense that every infinitesimal deformation of a multiple cover in $\CM_{v,d}(\Gamma^+,\Gamma^-)$ as a holomorphic curve is still a multiple cover.   

\subsection{Obstruction bundles from positivity of intersections}
We denote by $\IM^0_v=\IM^0_{v,d}(\Gamma^+,\Gamma^-)$ the moduli space of parametrized branched coverings $u=v\circ\varphi: (S,j)\to (S',j')\to(X,J)$ of a fixed holomorphic map $v:(S',j')\to(X,J)$,. Note that for here we do not yet divide out any symmetry of the domain. As in \cite{F2}, for establishing the desired isolatedness result we need to prove that $$T_u\IM^0_v=\ker D_u,$$ see also (\cite{MDSa}, section 7.2). Here $D_u: T_u\BB\to\EE_u$ denotes the linearization of the Cauchy-Riemann operator $\CR_J$, viewed as a smooth section in an appropriate Banach space bundle $ \EE\to\BB$ with fibre $\EE_u=L^{p,d}(\Lambda^{0,1}\otimes_{j,J}u^*TX)$ over the Banach manifold of maps $\BB$ with tangent space $T_u\BB=H^{1,p,d}(u^*TX)\oplus T_j\IM_{0,n}$, see \cite{BM}, where the second summand keeps track of the variation of the complex structure on $S$. \\

Note that we always have the inclusion $T_u\IM^0_v\subset\ker D_u$. Assume that the underlying holomorphic curve $v:(S',j')\to (X,J)$ is simple and $J$ is a generic almost complex structure on $X$. By the well-known transversality result for simple holomorphic curves, it follows that the local dimension of $\IM^0_v$ near $u=v\circ\varphi$ is given by $$\dim T_u\IM^0_v = \ind(v)+2 \#\Crit(\varphi),$$ where $\#\Crit(\varphi)$ denotes the number of branch points of the branched covering map $\varphi:(S,j)\to(S',j')$. Note that the latter number is fully determined by $\Gamma^+$ and $\Gamma^-$. \\

Since $\ind u\leq \dim\ker D_u$, it follows that the desired equality $T_u\IM^0_v=\ker D_u$ can only hold when $\ind u\leq \ind v + 2 \#\Crit(\varphi)$. While the index $\ind u=\dim \ker D_u-\dim \coker D_u$ can be computed from topological data, in particular,  is constant over each connected component of the moduli space of holomorphic curves, the dimensions of kernel and cokernel themselves usually jump and are very hard to be controlled in general. \\

It was shown in \cite{Wen} that in $\dim X=4$  the dimensions of $\ker D_u$ and $\coker D_u$ can be controlled by topological data making use of positivity of intersections. This was used to prove an automatic transversality result for so-called nicely-embedded holomorphic curves by showing that the topological bounds imply $\coker D_u=0$ for any choice of almost complex structure $J$. We will now show that the same inequalities can be used to prove the following improvement about multiple covers. \\

Recall that in a three-dimensional contact manifold all closed Reeb orbits are either elliptic or hyperbolic. Here an orbit is called hyperbolic if all eigenvalues of the linearized return map are real and elliptic if they lie on the unit circle in the complex plane. Furthermore a (simple) holomorphic curve $v$ is called rigid if the Fredholm index of $v$ is zero. 

\begin{theorem} 
Assume that the rigid holomorphic curve $v:S\to X$ is immersed and that all its asymptotic orbits are elliptic. Then if $\ind u\leq \dim T_u\IM^0_v = 2 \#\Crit(\varphi)$ we have $T_u\IM^0_v=\ker D_u$, i.e., every infinitesimal deformation of $u$ as a holomorphic curve is again a multiple cover of $v$.   
\end{theorem}

Before we give the proof using the results from \cite{Wen}, we remark that this result indeed contains the automatic transversality result from \cite{Wen} as a special case (for index zero curves and up to excluding odd hyperbolic orbits because of troubles with bad orbits). Indeed it follows directly from the definition of nicely-embedded curves in (\cite{Wen}, definition 4.12) that each such holomorphic curve is immersed and has only odd orbits, see (\cite{Wen}, section 4.3). When $u=v$, i.e., $\varphi$ is the trivial covering, then $\ker D_u=0(=\ind(u))$ implies that $u$ is regular. 
  
\begin{proof} As mentioned above, we show the desired result using the following results in the paper \cite{Wen}. First, in (\cite{Wen}, section 3.3) it is shown that for every holomorphic curve $u:(S,j)\to (X,J)$ there exists a splitting of the pull-back bundle $$u^*TX=T_u\oplus N_u,$$ where $T_u$ and $N_u$ denote the tangent and the normal bundle to $u$, respectively. While $T_u$ agrees with $du(TS)$ away from the critical points of $u$, the normal bundle $N_u$ agrees with the contact hyperplane distributions $\xi^{\pm}=TV^{\pm}\cap J^{\pm}TV^{\pm}$ in the cylindrical ends $(\IR^{\pm}\times V^{\pm},J^{\pm})$ of $(X,J)$. \\ 

Using this splitting, in (\cite{Wen}, section 3.4) the author introduces the normal Cauchy-Riemann operator $$D_u^N: H^{1,p,d}(N_u)\oplus T_j\IM_{0,n}\to L^{p,d}(\Lambda^{0,1}\otimes_{j,J}N_u).$$ Furthermore he shows in (\cite{Wen}, theorem 3) that the dimensions of kernel and cokernel of $D_u$ and $D_u^N$ are related by 
$$\dim\ker D_u=2\#\Crit(u)+\dim\ker D_u^N\;\textrm{and}\;\coker D_u = \coker D_u^N,$$ where $\Crit(u)$ denotes the set of critical points of $u$. \\
 
Since the underlying simple curve $v$ is immersed, we have $$\#\Crit(u)=\#\Crit(\varphi)=\dim T_u\IM^0_v.$$ Hence it remains to prove that $\ker D_u^N=0$, which can be shown using the estimates in \cite{Wen}. \\

First, following (\cite{Wen}, section 1.1) the normal first Chern number of $u$ is defined by $$2c_N(u)=\ind u - 2 + 2g + \#\Gamma_0 = \ind u - 2,$$ where $\#\Gamma_0=0$ denotes the number of even asymptotic orbits of $u$ and $g$ denotes the genus of $u$. Note that the number of even asymptotic orbits of $u$ is indeed zero as every multiply-covered elliptic orbit is still elliptic and hence odd. \\

Following (\cite{Wen}, proposition 3.18) the normal first Chern number is related to an adjusted first Chern number $c_1(N_u)$ of the normal bundle by $$c_1(N_u)=c_N(u)-2\#\Crit(u),$$ where again $\Crit(u)=\Crit(\varphi)$ agrees with the set of branch points. Here ''adjusted'' refers to the fact that the count of zeros involves asymptotic intersections. \\ 

Since $\ind(u)\leq \dim T_u\IM^0_v = 2\#\Crit(\varphi)$, note that as in the proof of (\cite{Wen}, theorem 1) we still have $c_1(N_u)<0$, independent of the number of branch points. While the definition of $c_1(N_u)$ is quite complicated, all we need for our proof is that it is shown in (\cite{Wen}, proposition 2.2 (1)) that the latter implies $\ker D_u^N=0$. \end{proof}  

In particular, the latter holds true for all multiple covers of immersed holomorphic curves $v$ with elliptic orbits which are virtually rigid, i.e., with $\ind u=0$. As in (\cite{MDSa}, section 7.2) we can use our bound on the dimension of $\ker D_u$ to prove the existence of an obstruction bundle of constant rank. 

\begin{corollary} If $\ind u=0$ for $u\in \IM_{v,d}(\Gamma^+,\Gamma^-)$, then the cokernels of the linearized Cauchy-Riemann operator $\CR_J$ fit together to a smooth \emph{obstruction bundle} $$\Coker_v\CR_J=\Coker_{v,d}\CR_J(\Gamma^+,\Gamma^-)$$ over the moduli space of multiple covers $\IM_v=\IM_{v,d}(\Gamma^+,\Gamma^-)$. \end{corollary} 

\begin{proof} From the desired equality $T_u\IM^0_v=\ker D_u$ it follows that the cokernel $\coker D_u$ has dimension 
$$\dim\coker D_u = \dim\ker D_u - \ind u = \ind v + 2 \#\Crit(\varphi) - \ind u,$$ in particular, is constant over each connected component of the moduli space of multiple covers. \end{proof} 

In \cite{F2} the author has shown that such an obstruction bundle exists over the moduli space of multiple covers over each orbit cylinder $\IR\times\gamma$ in $\IR\times V$. The neccessary equality $T_u\IM^0_{\gamma}=\ker D_u$ for all moduli spaces $\IM^0_{\gamma }=\IM^0_{\gamma,d}(\Gamma^+,\Gamma^-)$ of orbit curves was proven in \cite{F2} using energy considerations, since it is the infinitesimal version of the statement in \cite{BEHWZ} that every holomorphic curve with zero contact area (in the sense of \cite{BEHWZ}) is a multiple cover of an orbit cylinder. \\

Further it follows from the compactness results from \cite{BEHWZ} stated above (see also \cite{F2}) that the codimension-one boundary of the compactified moduli space $\CM_{\gamma,d}(\Gamma^+,\Gamma^-)$ of multiple covers of the orbit cylinder $\IR\times\gamma$ is given by (fibre) products $\CM_{\gamma,d}(\Gamma_1^+,\Gamma_1^-)\times\CM_{\gamma,d}(\Gamma_2^+,\Gamma_2^-)$ of lower-dimensional moduli spaces of multiple covers over the same orbit cylinder. Note that here we work with the unperturbed moduli space, which agrees with the moduli space for the contact manifold $S^1$, so everything reduces to studying branched coverings of cylinders and all curves are automatically regular (as a map to the cylinder over $S^1$). \\

The same compactness result proves that the codimension-one boundary of the compactified moduli space $\CM_{v,d}(\Gamma^+,\Gamma^-)$ of multiple covers over a rigid immersed curve with elliptic orbits $v$ is given by (fibre) products $\CM_{\gamma^+,d^+}(\Gamma_1^+,\Gamma_1^-)\times\CM_{v,d}(\Gamma_2^+,\Gamma_2^-)$ and $\CM_{v,d}(\Gamma_1^+,\Gamma_1^-)\times\CM_{\gamma^-,d^-}(\Gamma_2^+,\Gamma_2^-)$ of lower-dimensional moduli spaces, where on the zero level we still consists of multiple covers over the original rigid immersed curve with elliptic orbits, while on the positive and negative levels we find multiple covers over cylinders over asymptotic Reeb orbits $\gamma^+$ and $\gamma^-$ for $v$, respectively. \\

In order to compute the new SFT invariants for closed Reeb orbits, it was further shown in \cite{F2} that the obstruction bundle actually extends to the compactified moduli space, which again follows from energy reasons and a linear gluing theorem. We now prove that also the above obstruction bundle $\Coker_v\CR_J=\Coker_{v,d}(\Gamma^+,\Gamma^-)$ over the moduli space of multiple covers of a rigid immersed curves with only elliptic asymptotic orbits extends to the compactification.

\begin{theorem} 
The obstruction bundle $\Coker_v\CR_J=\Coker_{v,d}(\Gamma^+,\Gamma^-)$ over the moduli space $\IM_v=\IM_{v,d}(\Gamma^+,\Gamma^-)$ of index zero multiple covers of a rigid immersed curve with elliptic orbits extends to a smooth bundle $\overline{\Coker}_v\CR_J=\overline{\Coker}_{v,d}(\Gamma^+,\Gamma^-)$ over its compactification $\CM_v=\CM_{v,d}(\Gamma^+,\Gamma^-)$.
\end{theorem}

\begin{proof} Assume that the sequence of multiple covers $u_n=v\circ\varphi_n\in\IM_{v,d}(\Gamma^+,\Gamma^-)$ converges to a two-level curve $(u^+,u^0)\in\IM_{\gamma^+,d^+}(\Gamma_1^+,\Gamma_1^-)\times\IM_{v,d}(\Gamma_2^+,\Gamma_2^-)$ in the sense of \cite{BEHWZ}. Following \cite{F2} it can be shown using work of Long about the Conley-Zehnder index of multiply-covered Reeb orbits that $\ind u^+\geq 0$ when $\dim X=4$. Together with $\ind u^++\ind u^0=\ind u=0$ we get $\ind u^0\leq 0$, so that $u^0$ still meets the requirements from the above theorem and we have $T_{u^0}\IM^0_v=\ker D_{u^0}$. On the other hand, from \cite{F2} we get by energy considerations that $T_{u^+}\IM^0_{\gamma^+}=\ker D_{u^+}$. Putting together, we find that also for the broken curve $(u^+,u^0)\in\IM_{\gamma^+}\times\IM_v$ in the compactification we have $$T_{(u^+,u^0)}(\IM^0_{\gamma^+}\times\IM^0_v)=T_{u^+}\IM^0_{\gamma^+}\oplus T_{u^0}\IM^0_v=\ker D_{u^+}\oplus \ker D_{u^0} = \ker D_{(u^+,u^0)}$$ as desired, so that we can prove the existence of the smooth bundle $\overline{\Coker}_v\CR_J$ using linear gluing as in \cite{F2}. \end{proof}
  
\subsection{Transversality for multiple covers using obstruction bundles}
After showing that for the moduli spaces of multiple covers of rigid immersed curves with elliptic orbits we have the same nice obstruction bundle setup as for the moduli spaces of multiple covers over orbit cylinders, we now want to discuss the appearing algebraic structures. In other words, we want to discuss in how far one can actually count multiple covers of these nice curves, i.e., in how far this count depends on chosen auxiliary data like abstract perturbations needed in order to achieve regularity. \\

Before all that, we however need to state the main theorem about obstruction bundle transversality, see (\cite{F2}, proposition 3.1). For alternative proofs we refer to (\cite{MDSa}, proposition 7.2.3) and the proof of the main theorem in \cite{LP}. 

\begin{proposition} Let $\nu$ be a section in the cokernel bundle $\Coker\CR_J\subset\EE|_{\IM}$ over the moduli space $\IM=\CR_J^{-1}(0)\subset\BB$, which is extended (using parallel transport and cut-off functions, as described in \cite{F2},\cite{MDSa},\cite{LP}) to a section in the full Banach space bundle $\EE\to\BB$. Then it holds:
\begin{itemize} 
\item The perturbed moduli space $\IM^{\nu}=(\CR_J+\nu)^{-1}(0)$ agrees with the zero set of $\nu$ in $\IM$,   
\begin{equation*} \IM^{\nu}=\nu^{-1}(0). \end{equation*}
\item If $\nu$ is a transversal section in $\Coker\CR_J$, then $\CR^{\nu}_J$ is a transversal section in $\EE$, i.e., $\IM^{\nu}$ is regular.
\item The linearization of $\nu$ at every zero is a compact operator, so that the linearizations of $\CR_J$ and $\CR_J^{\nu}$ belong to the same class of Fredholm operators.
\end{itemize}
\end{proposition}

Note that the above theorem does not only hold for the non-compact moduli space itself, but also for the other moduli spaces of orbit curves appearing in the compactification. Using the linear gluing result in \cite{F2} it follows that a compactification of the perturbed moduli space $\IM^{\nu}$ is given by $$\CM^{\bar{\nu}}=\bar{\nu}^{-1}(0)\subset\CM$$ for a smooth section $\bar{\nu}$ in the extended obstruction bundle $\overline{\Coker}\CR_J$ over the compactified nonregular moduli space $\CM$ of multiple covers of the orbit cylinder. Note that in order to formulate the corresponding version of the above proposition directly for the compactified moduli space, one has to work with polyfolds instead of Banach manifolds. \\

In Gromov-Witten theory we would hence obtain the contribution of the regular perturbed moduli space by integrating the Euler class of the finite-dimensional obstruction bundle over the compactified moduli space. On the other hand, passing from Gromov-Witten theory back to symplectic field theory again, we see that the presence of codimension-one boundary of the nonregular moduli spaces of branched covers implies that Euler numbers for sections in the cokernel bundles are not defined in general, since the count of zeroes depends on the compact perturbations chosen for the moduli spaces in the boundary. Instead of looking at a single moduli space, we have to consider all moduli spaces at once and define {\it coherent} collections of sections in the obstruction bundles $\overline{\Coker}\CR_J$ over all moduli spaces $\CM$. \\

We start with the case of multiple covers of orbit cylinders. Recall from \cite{F2} that the codimension-one boundary of every moduli space $\CM=\CM_{\gamma,d}(\Gamma^+,\Gamma^-)$ again consists of curves with two levels, whose moduli spaces can be represented as products $\CM_1\times\CM_2=\CM_{\gamma,d_1}(\Gamma_1^+,\Gamma_1^-)\times\CM_{\gamma,d_2}(\Gamma_2^+,\Gamma_2^-)$ of moduli spaces of strictly lower dimension, where the first index refers to the level. On the other hand, it follows from the linear gluing result in \cite{F2} that over the boundary component $\CM_1\times\CM_2$ the cokernel bundle $\overline{\Coker}\CR_{J}=\overline{\Coker}_{\gamma,d}\CR_{J}(\Gamma^+,\Gamma^-)$ is given by 
\begin{equation*} \overline{\Coker}\CR_{J}|_{\CM_1\times\CM_2} = \pi_1^*\overline{\Coker}^1\CR_{J}\oplus\pi_2^*\overline{\Coker}^2\CR_{J}, \end{equation*}
where $\overline{\Coker}^{1,2}\CR_{J}=\overline{\Coker}_{\gamma,d_{1,2}}\CR_{J}(\Gamma_{1,2}^+,\Gamma_{1,2}^-)$, denote the cokernel bundles over the compact moduli spaces $\CM_{1,2}=\CM_{\gamma,d_{1,2}}(\Gamma_{1,2}^+,\Gamma_{1,2}^-)$ and $\pi_{1,2}:\CM_1,\CM_2\to\CM_{1,2}$ is the projection onto the first or second factor, respectively. \\  

Assume that we have chosen sections $\bar{\nu}=\bar{\nu}_{\gamma,d}(\Gamma^+,\Gamma^-)$ in the cokernel bundles $\overline{\Coker}\CR_{J}$ over all moduli spaces $\CM$ of branched covers. Following (\cite{F2}, definition 3.2) we call this collection of sections $(\bar{\nu})$ {\it coherent} if over every codimension-one boundary component $\CM_1\times\CM_2$ of a moduli space $\CM=\CM_{\gamma,d}(\Gamma^+,\Gamma^-)$ the corresponding section $\bar{\nu}$ agrees with the pull-back $\pi_1^*\bar{\nu}_1\oplus\pi_2^*\bar{\nu}_2$ of the chosen sections $\bar{\nu}_{1,2}$ in the cokernel bundles $\overline{\Coker}^{1,2}\CR_{J}$ over $\CM_{1,2}$, respectively. \\

Since in the end we will again be interested in the zero sets of these sections, we will again assume that all occuring sections are transversal to the zero section. On the other hand, it is not hard to see that one can always find such coherent collections of (transversal) sections in the cokernel bundles by using induction on the dimension of the underlying nonregular moduli space of branched covers.\\

Now we want to turn to the moduli spaces of multiple covers of immersed curves with elliptic orbits in four-dimensional symplectic cobordisms. Recall that the codimension-one boundary of every moduli space of branched covers $\CM=\CM_{v,d}(\Gamma^+,\Gamma^-)$ again consists of curves with two levels, whose moduli spaces can be represented as products $\CM_+\times\CM_0=\CM_{\gamma^+,d^+}(\Gamma_1^+,\Gamma_1^-)\times\CM_{v,d}(\Gamma_2^+,\Gamma_2^-)$ and $\CM_0\times\CM_-=\CM_{v,d}(\Gamma_1^+,\Gamma_1^-)\times\CM_{\gamma^-,d^-}(\Gamma_2^+,\Gamma_2^-)$ of moduli spaces of strictly lower dimension. On the other hand, it follows that over the boundary component $\CM_+\times\CM_0$ or $\CM_0\times\CM_-$ the cokernel bundle $\overline{\Coker}\CR_{J}=\overline{\Coker}_{v,d}\CR_J(\Gamma^+,\Gamma^-)$ is given by 
\begin{eqnarray*} 
\overline{\Coker}\CR_{J}|_{\CM_+\times\CM_0} &=& \pi_+^*\overline{\Coker}^+\CR_{J}\oplus\pi_0^*\overline{\Coker}^0\CR_{J}, \\
\overline{\Coker}\CR_{J}|_{\CM_0\times\CM_-} &=& \pi_0^*\overline{\Coker}^0\CR_{J}\oplus\pi_-^*\overline{\Coker}^-\CR_{J}, 
\end{eqnarray*}
where  $\overline{\Coker}^0\CR_{J}=\overline{\Coker}_{v,d}\CR_J(\Gamma_{1,2}^+,\Gamma_{1,2}^-)$ and $\overline{\Coker}^+\CR_{J}=\overline{\Coker}_{\gamma^+,d^+}\CR_J(\Gamma_1^+,\Gamma_1^-)$, $\overline{\Coker}^-\CR_{J}=\overline{\Coker}_{\gamma^-,d^-}\CR_J(\Gamma_2^+,\Gamma_2^-)$ denote the cokernel bundle over the moduli space $\CM_0$ of multiple covers of the immersed curves with elliptic orbits and the moduli spaces $\CM_+$, $\CM_-$ of multiple covers of cylinders over positive or negative asymptotic Reeb orbits $\gamma^{\pm}$ of $v$, respectively. \\

With this we can now give the analogue of the above definition of special sections $\bar{\nu}=\bar{\nu}_{v,d}(\Gamma^+,\Gamma^-)$ in obstruction bundles $\overline{\Coker}\CR_{J}=\overline{\Coker}_{v,d}\CR_J(\Gamma^+,\Gamma^-)$ over moduli spaces $\CM=\CM_{v,d}(\Gamma^+,\Gamma^-)$ of multiple covers of immersed curves with elliptic orbits $v$. Assume that we have already coherently chosen sections $\bar{\nu}_{\pm}=\bar{\nu}_{\gamma^{\pm},d}(\Gamma^+,\Gamma^-)$ in the cokernel bundles $\overline{\Coker}^{\pm}\CR_{J}=\overline{\Coker}_{\gamma^{\pm},d}\CR_J(\Gamma^+,\Gamma^-)$ over all moduli spaces $\CM_{\pm}=\CM_{\gamma,d}(\Gamma^+,\Gamma^-)$ of branched covers of cylinders over positive and negative asymptotic Reeb orbits $\gamma^{\pm}$ of $v$.

\begin{definition} Assume that we have chosen sections $\bar{\nu}$ in the cokernel bundles $\overline{\Coker}\CR_{J}$ over all moduli spaces $\CM$ of multiple covers of the immersed curve with elliptic orbits $v$. Then we call such a collection of sections $(\bar{\nu})$ {\it coherently connecting} $(\bar{\nu}_+)$ and $(\bar{\nu}_-)$ if over every codimension-one boundary component $\CM_+\times\CM_0$, $\CM_0\times \CM_-$ the corresponding section $\bar{\nu}$ agrees with the pull-back $\pi_+^*\bar{\nu}_+\oplus\pi_0^*\bar{\nu}_0$,  $\pi_0^*\bar{\nu}_0\oplus\pi_-^*\bar{\nu}_-$ of the chosen sections $\bar{\nu}_0$ and $\bar{\nu}_+$, $\bar{\nu}_-$ in the cokernel bundles $\overline{\Coker}^0\CR_{J}$ over $\CM_0$, $\overline{\Coker}^{\pm}\CR_{J}$ over $\CM_{\pm}$, respectively.  \end{definition}

Since in the end we will again be interested in the zero sets of sections, we will again assume that all occuring sections are transversal to the zero section. On the other hand, it is again not hard to see that one can always find such coherent collections of (transversal) sections in the cokernel bundles by using induction on the dimension of the underlying nonregular moduli space of branched covers. 

\subsection{Counting multiple covers of immersed curves with elliptic orbits}
We now turn to the resulting algebraic structures, where we first recall the algebraic formalism to define invariants for closed Reeb orbits. Denote by $\PP_{\gamma}$ be the graded Poisson subalgebra of the Poisson algebra $\PP$ of rational SFT, which is generated only by those $p$- and $q$-variables $p_{\gamma^n}$, $q_{\gamma^n}$ corresponding to Reeb orbits which are multiple covers of the fixed orbit $\gamma$ and which are good in the sense of \cite{BM}.  It will become important that the natural identification of the formal variables $p_{\gamma^n}$ and $q_{\gamma^n}$ for different orbits $\gamma$ does {\it not} lead to an isomorphism of the graded algebras $\PP_{\gamma}$ with the corresponding graded algebra $\PP_{S^1}$ for $\gamma=V=S^1$, not only since the gradings of $p_{\gamma^n}$ and $q_{\gamma^n}$ are different and hence even the commutation rules may change but also that variables $p_{\gamma^n}$ and $q_{\gamma^n}$ may not be there since they would correspond to bad orbits. \\ 

In (\cite{EGH}, section 2.2.3) one collects the information about all moduli spaces of holomorphic curves in $\IR\times V$ in a generating function, the SFT Hamiltonian $\Ih$, which does not only depend on contact form and cylindrical almost complex structure but also on the collection of abstract perturbations. As in \cite{F2} we now define a local SFT Hamiltonian $\Ih_{\gamma}\in\PP_{\gamma}$  by only counting branched covers of the cylinder over the Reeb orbit $\gamma$. Instead of working with polyfold perturbations, we have seen above that we can make all moduli spaces of orbit curves regular by choosing sections in the cokernel bundles over all moduli spaces. For such a collection of sections $(\bar{\nu})$ we then define the Hamiltonian $\Ih_{\gamma}=\Ih_{\gamma}^{(\bar{\nu})}$ by 
$$\Ih_{\gamma}^{(\bar{\nu})}=\sum_{\Gamma^+,\Gamma^-} \frac{1}{s^+!s^-!\kappa_{\Gamma^+}\kappa_{\Gamma^-}}\#\CM^{\bar{\nu}}_{\gamma,d}(\Gamma^+,\Gamma^-) q^{\Gamma^-} p^{\Gamma^+},$$ with $p^{\Gamma^+}=p_{\gamma^{n_1^+}}\ldots p_{\gamma^{n_{s^+}^+}}$ and $q^{\Gamma^-}=q_{\gamma^{n_1^-}}\ldots q_{\gamma^{n_{s^-}^-}}$, where $$\CM^{\bar{\nu}}_{\gamma,d}(\Gamma^+,\Gamma^-)=\bar{\nu}^{-1}(0)\subset\CM_{\gamma,d}(\Gamma^+,\Gamma^-).$$ Furthermore $s^{\pm}=\#\Gamma^{\pm}$ and $\kappa_{\Gamma^{\pm}}=\kappa_{\gamma^{n_1^-}}\ldots \kappa_{\gamma^{n_{s^-}^-}}$, where $\kappa_{\gamma}$ denotes the multiplicity of the orbit $\gamma$. \\

Note that in general we have to expect that the local SFT Hamiltonian explicitly depends on the chosen coherent collection of sections. However, in \cite{F2} we were able to prove the following result.

\begin{theorem} For every closed Reeb orbit $\gamma$ the Hamiltonian $\Ih_{\gamma}= \Ih^{\bar{\nu}}_{\gamma}$ vanishes independently of the chosen coherent collection of sections $(\bar{\nu})$ in the cokernel bundles over all moduli spaces of branched covers, $\Ih_{\gamma} = \Ih^{\bar{\nu}}_{\gamma}= 0.$ \end{theorem}

Although the result of our computation may suggest that it follows a global symmetry of the resulting regular moduli space, we want to emphasize that the $S^1$-action on the underlying nonregular moduli space of branched covers in general does not lift to an action on the obstruction bundle over this space, so that the resulting perturbed moduli space does {\it not} carry a global symmetry. \\

In the same way as for a single orbit we define for collections of Reeb orbits $\Gamma$ the Poisson algebras $\PP_{\Gamma}$ to be the graded Poisson subalgebras of the Poisson algebra $\PP$, which is generated only by those $p$- and $q$-variables $p_{\gamma^n}$, $q_{\gamma^n}$ corresponding to Reeb orbits which are multiple covers of orbits $\gamma\in\Gamma$.
For a rigid immersed holomorphic curve $v$ with asymptotic orbits $\Gamma'^+$ and $\Gamma'^-$ let $\LL_{\Gamma'^+,\Gamma'^-}$ be the space of formal power series in the variables $p^+_{\gamma_+^n}$ with $\gamma^+\in\Gamma'^+$ with coefficients which are polynomials in the variables $q^-_{\gamma_-^n}$, $\gamma^-\in\Gamma'^-$. Furthermore we introduce as in \cite{EGH} the bigger space $\hat{\LL}_{\Gamma'^+,\Gamma'^-}$ whose elements are power series in $p^+_{\gamma_+^n}$ and $p^-_{\gamma_-^n}$ which are polynomials in $q^+_{\gamma_+^n}$ and $q^-_{\gamma_-^n}$. \\

While in standard SFT one collects the information about all moduli spaces of holomorphic curves in $X$ by defining a potential $\If$, we now define a local SFT potential $\If_v\in\LL_{\Gamma'^+,\Gamma'^-}$ counting only multiple covers of the fixed rigid immersed curve with elliptic orbits $v:S\to X$ as in (\cite{EGH}, section 2.3.2).

\begin{definition} For every choice of obstruction bundle sections $(\bar{\nu})$ coherently connecting the coherent collections of obstruction bundle sections $(\bar{\nu}^{\pm})$ chosen for all positive and negative asymptotic Reeb orbits $\gamma^{\pm}\in\Gamma'^{\pm}$ of $v$, we define the local SFT potential of a rigid immersed holomorphic curve $v$ with elliptic orbits by $$\If_v\stackrel{!}{=}\If_v^{(\bar{\nu})}=\sum_{\Gamma^+,\Gamma^-} \frac{1}{s^+!s^-!\kappa_{\Gamma^+}\kappa_{\Gamma^-}}\#\CM^{\bar{\nu}}_{v,d}(\Gamma^+,\Gamma^-) q_-^{\Gamma^-} p_+^{\Gamma^+},$$ where $\CM^{\bar{\nu}}_{v,d}(\Gamma^+,\Gamma^-)=\bar{\nu}^{-1}(0)\subset\CM_{v,d}(\Gamma^+,\Gamma^-)$. \end{definition} 

For the rest of this section we want to discuss in how far the local SFT Hamiltonians $\Ih_{\gamma}$ and the local SFT potentials $\If_v$ depend on the collections of obstruction bundle sections $(\bar{\nu})$ needed to define them. While the above result about the local SFT Hamiltonian makes identities like the master equation $\{\Ih_{\gamma},\Ih_{\gamma}\}=0$ trivial, we can use it to prove the following important result

\begin{theorem} Assume that the coherent collections of sections $(\bar{\nu}_{\pm})$ are fixed for all asymptotic Reeb orbits $\gamma^{\pm}\in\Gamma'^{\pm}$ of $v$. Then the local SFT potential $\If_v=\If_v^{(\bar{\nu})}$ of $v$ is independent of the chosen collection of sections $(\bar{\nu})$ coherently connecting $(\bar{\nu}^+)$ and $(\bar{\nu}^-)$. In particular, the algebraic count of multiple covers of the immersed curve with elliptic orbits $v$ is well-defined. \end{theorem}

\begin{proof} For two collections of sections $(\bar{\nu}_0)$ and $(\bar{\nu}_1)$ coherently connecting $(\bar{\nu}^+)$ and $(\bar{\nu}^-)$ let $(\bar{\nu}_s)$, $s\in [0,1]$ be the family of coherently connecting collections of sections given by $\bar{\nu}_s=(1-s)\cdot\bar{\nu}_0+s\cdot\bar{\nu}_1$, which defines a transversal section $\bar{\nu}$ in $\overline{\Coker}\CR_J$ over $\CM\times [0,1]$, possibly after small perturbation. Let $\If_v^s=\If_v^{(\bar{\nu}_s)}$ denote the corresponding family of local SFT potentials of $v$. Then it follows from the proof of the corresponding result for the standard SFT potential in (\cite{EGH}, section 2.4) that $\If_v^0$ and $\If_v^1$ are homotopic through the homotopy $\If_v^s$ in the sense of \cite{EGH}, i.e., there exists another family $k_v^s$ such that the family $s\mapsto \If_v^s$ satisfies the following Hamilton-Jacobi equation in $\hat{\LL}_{\Gamma'^+,\Gamma'^-}$, 
$$\frac{\del\If_v^s(p^+,q^-)}{\del s}=\IG_v\Bigl(p^+,\frac{\del\If_v^s(p^+,q^-)}{\del p^+}, \frac{\del\If_v^s(p^+,q^-)}{\del q^-}, q^-\Bigr),$$ where 
\begin{eqnarray*}
\IG_v(p^+,q^+,p^-,q^-)&=&\{\Ih_{\Gamma'^+} - \Ih_{\Gamma'^-},k_v^s\}\\&=&
\sum_{\gamma^{\pm}\in\Gamma'^{\pm}}\kappa_{\gamma^-}\frac{\del\Ih_{\Gamma'^-}(p^-,q^-)}{\del p_{\gamma^-}^-}
\frac{\del k_v^s(p^+,q^-)}{\del q_{\gamma^-}^-}\\&+&\kappa_{\gamma^+}\frac{\del k_v^s(p^+,q^-)}{\del p_{\gamma^+}^+}\frac{\del\Ih_{\Gamma'^+}(p^+,q^+)}{\del q_{\gamma^+}^+}.
\end{eqnarray*}
Since $\Ih_{\Gamma'^{\pm}}=\sum_{\gamma^{\pm}\in\Gamma'^{\pm}}\Ih_{\gamma^{\pm}}=0$,  as by the above theorem $\Ih_{\gamma^{\pm}}=\Ih_{\gamma^{\pm}}^{(\bar{\nu}^{\pm})}=0$ for all closed Reeb orbits $\gamma^{\pm}$ and all coherent collections of sections $(\bar{\nu}^{\pm})$, it follows that $\IG_v=0$, so that $\If_v^s$ must be independent of $s\in[0,1]$, in particular, $\If_v^0=\If_v^1\in\LL_{\Gamma'^+,\Gamma'^-}$. \end{proof}

We end this subsection by discussing how the local SFT potential $\If_v$ depends on the choice of coherent collections of sections $(\bar{\nu}^{\pm})$ for all closed Reeb orbits $\gamma^{\pm}\in \Gamma'^{\pm}$, where it will turn out that $\If_v$ indeed depends on this choice. Let $\LL_-$, $\LL_+$, $\LL$ be generated by $(q^-,p)$-, $(q,p^+)$- and $(q^-,p^+)$-variables, respectively. Following (\cite{EGH}, section 2.5) we define the operation $\#:\LL_-\times\LL_+\to\LL$ by 
$$(f_-\# f_+)(q^-,p^+)=(f_-(q^-,p)+f_+(q,p^+)-\sum_{\gamma}\kappa_{\gamma}^{-1} q_{\gamma} p_{\gamma})|_L$$ for $f_{\pm}\in\LL_{\pm}$, where $L$ is the Lagrangian in the symplectic super-space spanned by $(q^-,p^+)$-variables which is determined by the equations $$q_{\gamma}=\kappa_{\gamma}\frac{\del f_-}{\del p_{\gamma}},\;\; p_{\gamma}=\kappa_{\gamma}\frac{\del f_+}{\del q_{\gamma}}.$$ For chosen homotopies $(\bar{\nu}_{01}^+)$, $(\bar{\nu}_{10}^-)$ from $(\bar{\nu}_0^+)$ to $(\bar{\nu}_1^+)$, $(\bar{\nu}_1^-)$ to $(\bar{\nu}_0^-)$, respectively, let $\If_{\Gamma'^+}^{01}=\If_{\Gamma'^+}^{(\bar{\nu}_{01}^+)}\in$, $\If_{\Gamma'^-}^{10}=\If_{\Gamma'^-}^{(\bar{\nu}_{10}^-)}$ be the local SFT potential of the union of orbit cylinders in the cylindrical manifold equipped with non-cylindrical data. Note that it again follows from the above theorem that the local SFT potentials $\If_{\Gamma'^+}^{01}$ and $\If_{\Gamma'^-}^{10}$ are independent of the chosen homotopies. With this we can describe the change of the local SFT potential $\If_v$ under different choices of coherent collections of sections $(\bar{\nu}^{\pm})$ as follows. 

\begin{theorem} For two different choices of coherent collections of sections $(\bar{\nu}_0^{\pm})$ and $(\bar{\nu}_0^{\pm})$ we denote by $\If_v^0=\If_v^{(\bar{\nu}_0^{\pm})},\If_v^1=\If_v^{(\bar{\nu}_1^{\pm})}\in\LL_{\Gamma'^+,\Gamma'^-}$. Then we have $$\If_v^1= \If_{\Gamma'^-}^{10}\# \If_v^0\# \If_{\Gamma'^+}^{01}.$$ \end{theorem}

Note that this theorem follows by combining the algebraic formalism for composition of cobordisms in (\cite{EGH}, section 2.5) with our above result stating that $\If_v$ is independent of the chosen collection of sections coherently connecting two coherent collections of sections. On the other hand, we want to emphasize that from the above result one can deduce as in \cite{EGH} a functoriality as known from Floer homology. To this end, observe that the local SFT potential $\If_v\in\LL_{\Gamma'^+,\Gamma'^-}$ defines a Lagrangian $L_{\If_v}$ in the symplectic super-space spanned by the $p^{\pm}$- and $q^{\pm}$-variables, $$L_{\If_v}=\{q^+=\frac{\del\If_v}{\del p^+}, p^-=\frac{\del\If_v}{\del q^-}\}. $$ Viewing functions in $\PP^{\pm}_{\Gamma^{\pm}}$ in the natural way as elements in the bigger space $\tilde{\LL}_{\Gamma'^+,\Gamma'^-}$ we follow \cite{EGH} and define maps $$f_v^{\pm}:\PP_{\Gamma'^{\pm}}\to\LL_{\Gamma'^+,\Gamma'^-},\; g\mapsto g|_{L_{\If_v}}.$$  It is shown in \cite{EGH} that the local SFT potentials $\If_{\Gamma'^+}^{01}$ and $\If_{\Gamma'^-}^{10}$ define not only automorphisms of $\PP_{\Gamma'^+}$ and $\PP_{\Gamma'^-}$, respectively, but also an automorphism of $\LL_{\Gamma'^+,\Gamma'^-}$. We get the following functorial property of the maps $f_v^{\pm}$. 

\begin{corollary} After applying the automorphisms of $\PP_{\Gamma'^+}$, $ \PP_{\Gamma'^-}$ and $\LL_{\Gamma'^+,\Gamma'^-}$ induced by $\If_{\Gamma'^+}^{01}$ and $\If_{\Gamma'^-}^{10}$, the map $f_v^{0,\pm}: \PP_{\Gamma'^{\pm}}\to\LL_{\Gamma'^+,\Gamma'^-}$ gets replaced by the map $f_v^{1,\pm}: \PP_{\Gamma'^{\pm}}\to\LL_{\Gamma'^+,\Gamma'^-}$. \end{corollary}

\section{Application: Stable hypersurfaces intersecting exceptional spheres}

Instead of discussing the full TQFT picture involving splitting and gluing of the underlying immersed curves with elliptic orbits, in this section we will show how local SFT methods can be applied to embedding problems in symplectic geometry. More precisely, we will show that every stable hypersurface which intersects an exceptional sphere in a closed four-dimensional symplectic manifold in a homological nontrivial way must carry an elliptic orbit. \\
 
\subsection{Additional marked points and gravitational descendants}
For this we use that a closed rigid nicely-embedded curve $v:(S^2,i)\to(X,J)$ is an \emph{exceptional sphere}. Indeed, it follows from the definitions that $\ind(v)=0$ and 
$\delta(v)=0$, so that $2c_N(v)=\ind v - 2 + 2g + \#\Gamma_0=-2$ and hence (with the absence of asymptotic intersections) $[v]\cdot[v]=i(v,v)=2\delta(v)+c_N(v)=-1$. 
On the other hand, since $c_1(v)=c_1(v^*TX)=c_1(N_v)+c_1(TS^2)=c_N(v)+2=1$ we get that the Fredholm index of a $d$-fold multiple cover $u=v\circ\varphi$ in the moduli space 
$\CM_{v,d}$ is given by $\ind(u)=-2+2c_1(u)=-2+2d=2(d-1)$ and hence strictly positive for $d>1$. \\

In order to get interesting contributions, it follows that we need to enrich 
the local Gromov-Witten potential by introducing marked points on the curve. As in standard Gromov-Witten theory these can be used to pull-back cohomology classes from the target 
which in turn can be integrated over the moduli space. On the other hand, since all maps $u$ in $\CM_{v,d}$ factor through $v$, $u=v\circ\varphi$, it follows that we can only 
expect to get non-zero integrals when the degree of the form is two or less. Since adding one additional marked point enlarges the dimension of the moduli space by two, it follows 
that this way we will do not get contributions from higher-dimensional moduli spaces. We solve this problem by additionally introducing gravitational descendants. \\

Let $\CM_{v,d,r}$ denote the moduli space of $d$-fold coverings $u=v\circ\varphi$ carrying $r$ additional marked points. In order to save notation, instead of integrating the 
pull-back of the canonical two-form on the sphere over the moduli space, we directly want to assume that every additional marked point $z_i$ gets mapped to a special marked point 
$p_i$ on the exceptional sphere under the covering map $\varphi$. Note that as an immediate consequence of the divisor equation in Gromov-Witten theory we get that $\CM_{v,d,r}$ 
is given by $d^r$ copies of the moduli space $\CM_{v,d}=\CM_{v,d,0}$ without additional marked points, where $d^r$ is the number of preimages of the special points $p_1,\ldots,p_r$ 
under the $d$-fold covering $\varphi$. \\

On the other hand, with the help of the additional marked points we can introduce $r$ tautological line 
bundles $\LL_1,\ldots,\LL_r$ over each moduli space $\CM_{v,d,r}$. They are defined as the pull-back of the vertical cotangent line bundle of 
$\pi_i: \CM_{v,d,r+1}\to\CM_{v,d,r}$ under the canonical section $\sigma_i: \CM_{v,d,r}\to\CM_{v,d,r+1}$ mapping to the $i$-th marked point in the fibre. It follows that the 
fibre $\LL_i$ over a smooth curve $(u,z_1,\ldots, z_r)$ is given by the cotangent line to the underlying Riemann sphere at the $i$.th marked point, 
$(\LL_i)_{((u,z_1,\ldots, z_r)}=T_{z_i}S^2$. With this the local Gromov-Witten potential can be enriched by integrating products $\psi_1^{j_1}\wedge\ldots\wedge\psi_r^{j_r}$ 
of powers of the first Chern classes $\psi_i=c_1(\LL_i)$, $i=1,\ldots,r$ over the moduli spaces $\CM_{v,d,r}$. It follows from the work in \cite{OP} for the case when the 
target manifold is a complex curve that the divisor $\CM^{(j_1,\ldots,j_r)}_{v,d,r}\subset\CM_{v,d,r}$ Poincare-dual to $\psi_1^{j_1}\wedge\ldots\wedge\psi_r^{j_r}$ has a 
geometric interpretation in terms of branching conditions. \\

\subsection{Obstruction bundle = normal bundle using topological recursion}
Instead of discussing the general statement, from now on let us restrict to the simplest non-trivial case $d=2$, $r=1$, $j=1$. Here it follows that the submoduli space 
$\CM_{v,2,1}^1\subset\CM_{v,2,1}$ consists of two-fold coverings $u=v\circ\varphi$ with one marked point mapping to the special point on the exceptional sphere which is 
additionally required to be a branch point of $\varphi$. While the (real) dimension of the (unperturbed) moduli space is two which accounts for the second branch point of 
the covering map $\varphi: S^2\to S^2$, the expected dimension of the moduli space $\CM_{v,2,1}^1$ given by the Fredholm index is $2(2-1)-2=0$. Apart from the fact that 
transversality for the Cauchy-Riemann operator cannot be satisfied, it follows our work from the last section there is an obstruction bundle $\overline{\Coker}\CR_J$ over 
$\CM_{v,2,1}^1$ of rank two. \\

Since our curves have no punctures and hence there is no codimension-one boundary of the moduli space, it follows that the count of elements in 
the resulting perturbed moduli space $(\CM_{v,2,1}^1)^{\bar{\nu}}=\bar{\nu}^{-1}(0)\subset\CM_{v,2,1}^1$ is independent of the chosen section $\bar{\nu}$ in $\overline{\Coker}\CR_J$.
Using topological recursion relations for descendants in Gromov-Witten theory we will show the count of elements $\#\CM_{v,2,1}^{1,\bar{\nu}}$ is (up to a combinatorical factor 
coming from the divisor equation) given by homological self-intersection number of the exceptional sphere. 

\begin{theorem} We have $\#(\CM_{v,2,1}^1)^{\bar{\nu}}=-\frac{1}{4}$. \end{theorem}

\begin{proof} The idea of the proof is that, using the topological recursion relations of Gromov-Witten theory, we can relate the above moduli space to the moduli space 
of doubly-covered spheres with one node. Since both components need to be simply-covered spheres which are automatically regular, the transversality problem in the Banach 
space bundle localizes on the nodal coincidence relation and hence reduces to geometric transversality. Since we need three marked points to apply topological recursion relations, 
we first apply the divisor equation twice to get $\#(\CM_{v,2,1}^1)^{\bar{\nu}}=\frac{1}{4}\cdot\#(\CM_{v,2,3}^{(1,0,0)})^{\bar{\nu}}$. Applying topological recursion relations we get 
that $$\#\CM_{v,2,3}^{(1,0,0),\bar{\nu}}=\#(\CM_{v,1,3}\times_{\ev}\CM_{v,1,2})^{\bar{\nu}}=\#(\CM_{v,1,1}\times_{\ev}\CM_{v,1,1})^{\bar{\nu}}$$ where 
$$\CM_{v,1,1}\times_{\ev}\CM_{v,1,1}=\{((u_1,w_1),(u_2,w_2)): u_1(w_1)=u_2(w_2)\}$$ and the second equality follows by applying the divisor equation in the reverse direction. 
Since we can assume after applying an automorphism of the domain that the simple covering $\varphi:S^2\to S^2$ in $u=v\circ\varphi$ is the identity and hence $u=v$, note that 
we can identify $\CM_{v,1,1}$ and hence also $\CM_{v,1,1}\times_{\ev}\CM_{v,1,1}$ with $S^2$ via $(u,z)=:z$. We now want to show that the obstruction bundle over 
$\CM_{v,1,1}\times_{\ev}\CM_{v,1,1}\cong S^2$ is given by the normal bundle to the exceptional sphere. \\
$ $\\
\emph{Claim:} $\overline{\Coker}\CR_J\cong N_v$. \\

For this we make use of the fact that by proving transversality of the Cauchy-Riemann operator in the Banach space bundle $\EE\oplus\EE$ over the Banach submanifold 
$$\BB\times_{\ev}\BB=\{(u_1,w_1),(u_2,w_2):u_1(w_1)=u_2(w_2)\}\subset\BB\times\BB$$ containing $\IM\times_{\ev}\IM$ ($\IM=\IM_{v,1,1}$) we do not only get transversality for 
the Cauchy-Riemann operator in $\EE\oplus\EE$ over the Banach manifold $\BB\times\BB$ of disconnected curves, but we also get that the evaluation map $\ev:\IM\times\IM\to X\times X$, 
$((u_1,w_1),(u_2,w_2))\mapsto(u_1(w_1),u_2(w_2))$ is transversal to the diagonal in $X\times X$ (see \cite{F3} for a proof of this lemma). It follows that the fibre of the cokernel bundle at 
$z\in S^2\cong\CM\times_{\ev}\CM$ is given by $(\overline{\Coker}\CR_J)_z=\coker D^N_z$, where $D^N_z$ is the restriction of the componentwise linearization in the normal direction 
$$D_N: H^{1,p}(N_v)\oplus H^{1,p}(N_v)\to L^p(\Lambda^{0,1}\otimes_{i,J} N_v)\oplus L^p(\Lambda^{0,1}\otimes_{i,J} N_v)$$ to the subspace 
$\{(\xi_1,\xi_2)\in H^{1,p}(N_v)\oplus H^{1,p}(N_v):\xi_1(z)=\xi_2(z)\}$. Since both components are simple and hence regular, and hence $D_N$ is an isomorphism, it follows that 
\begin{eqnarray*}
\coker D^N_z &\cong& \frac{H^{1,p}(N_v)\oplus H^{1,p}(N_v)}{\{(\xi_1,\xi_2)\in H^{1,p}(N_v)\oplus H^{1,p}(N_v):\xi_1(z)=\xi_2(z)\}}\\
             &\cong& \frac{(N_v)_z\oplus (N_v)_z}{\Delta} \cong (N_v)_z. 
\end{eqnarray*}
Putting everything together we get $$\#(\CM\times_{\ev}\CM)^{\bar{\nu}}=\int_{\CM\times_{\ev}\CM} e(\overline{\Coker}\CR_J)=\int_{S^2} e(N_v)=[v]\cdot[v]=-1.$$
\end{proof}

\subsection{Equations for the local SFT potentials}
We now want to turn again from local Gromov-Witten theory to local SFT. To this end we consider a stable hypersurface $V$ in the symplectic manifold $X$ which intersects the exceptional sphere $\Sigma:=v(S^2)$. Assuming that this intersection is homologically non-trivial in the sense that the union of circles $C=\Sigma\cap V$ defines a non-zero class in $H_1(V)$, it follows that after neck-stretching along $V$  (see \cite{BEHWZ}) the closed holomorphic sphere $v$ breaks up into two punctured holomorphic curves $v^+$ and $v^-$ (possibly with several connected components), connected by a collection $\Gamma$ of closed Reeb orbits on $V$ in the sense that $\Gamma$ is the set of the negative or positive asymptotic orbits of the punctured holomorphic curves $v^+$ and $v^-$, respectively. \\

For notational simplicity let us assume that $\Gamma$ just consists of a single closed Reeb orbit $\gamma$ and $V$ is separating, $X=X^+\cup_V X^-$, $V=\mp X^{\pm}$. Then the holomorphic sphere $v:(S^2,i)\to(X,J)$ breaks up into two holomorphic planes $v^{\pm}:(\IC,i)\to(X^{\pm},J^{\pm})$. Note that we continue not to distinguish between the compact symplectic manifolds with boundary $X^{\pm}$ and their completions $X^{\pm}\cup \IR^{\mp}\times V$ which are symplectic manifolds with cylindrical ends in the sense of \cite{BEHWZ}. On the other hand, since $\ind(v)=0$, we get from index additivity and regularity that $\ind(v^+)=\ind(v^-)=0$. \\

For the moment let us further assume that $v^+$ and $v^-$ are again immersed curves with elliptic orbits and that $\gamma$ is elliptic. We want to use our computation for the local Gromov-Witten potential of $v$ to prove results about the moduli spaces $\CM_{v^+,2}(\emptyset,\Gamma)$ and $\CM_{v^-,2}(\Gamma,\emptyset)$ from local SFT, where it turns out that the result depends on the behaviour of the Conley-Zehnder index for the multiple covered orbits $\gamma^k$. Recall that for every elliptic orbit $\gamma$ there exists an irrational number $\theta$ such that for the Conley-Zehnder indices we have $\CZ(\gamma^k)=2[k\theta]+1$, where $[x]$ denotes the largest integer less or equal than $x$. It follows that $\CZ(\gamma^2)-2\CZ(\gamma)=2([2\theta]-2[\theta])-1\in\{-1,+1\}$. Introducing additional marked points and gravitational descendants as in local Gromov-Witten theory using branching conditions to define moduli spaces $\CM_{v^+,2,1}^1(\emptyset,\Gamma)$ and $\CM_{v^-,2,1}^1(\Gamma,\emptyset)$, we now prove the following theorem.

\begin{theorem} If $\CZ(\gamma^2)-2\CZ(\gamma)=-1$ then 
$$\#\CM_{v^-,2}^{\bar{\nu}}(\gamma^2,\emptyset)+\#(\CM_{v^+,2,1}^1)^{\bar{\nu}}(\emptyset,(\gamma,\gamma)) =\#(\CM_{v^-,2,1}^1)^{\bar{\nu}}((\gamma,\gamma),\emptyset)=  -\frac{1}{4};$$
if $\CZ(\gamma^2)-2\CZ(\gamma)=+1$ then
$$\#\CM_{v^+,2}^{\bar{\nu}}(\emptyset,\gamma^2) +\#(\CM_{v^-,2,1}^1)^{\bar{\nu}}((\gamma,\gamma),\emptyset) =\#(\CM_{v^+,2,1}^1)^{\bar{\nu}}(\emptyset,(\gamma,\gamma))=  -\frac{1}{4}.$$
In particular, while the summands on the left side depend on the choice of coherent obstruction bundle sections $(\bar{\nu})$ for $\gamma$, the sum is independent of this choice.
\end{theorem}

\begin{proof} Let $u_n=v\circ\varphi_n$ be a sequence of multiple covers of the exceptional sphere. After neck-stretching along the hypersurface $V$ it follows from the 
compactness result in \cite{BEHWZ} that a subsequence converges to broken holomorphic curve $(u^+,u^-)$, which are multiple covers of the holomorphic planes $v^+$, $v^-$, 
respectively, $u^{\pm}=v^{\pm}\circ\varphi^{\pm}$. It follows that via compactness and gluing the moduli space $\CM_{v,2}$ is related to the union of moduli spaces 
$\bigcup_{\Gamma}\CM_{v^+,2}(\emptyset,\Gamma)\times\CM_{v^-,2}(\Gamma,\emptyset)$ of possibly disconnected curves, where $\Gamma=(\gamma,\gamma)$ or $\Gamma=\gamma^2$. 
It follows that, depending on whether we choose the special point on $\Sigma=v(S^2)$ on $\Sigma^+=\Sigma\cap X^+$ or $\Sigma^-=\Sigma\cap X^-$, we get that $\CM_{v,2,1}^1$  
is related to the union of moduli spaces $\bigcup_{\Gamma}\CM^1_{v^+,2,1}(\emptyset,\Gamma)\times\CM_{v^-,2}(\Gamma,\emptyset)$ or 
$\bigcup_{\Gamma}\CM_{v^+,2}(\emptyset,\Gamma)\times\CM^1_{v^-,2,1}(\Gamma,\emptyset)$. \\

In the case when $\Gamma=(\gamma,\gamma)$ note that the curves $u^+$ in 
$\CM_{v^+,2}(\emptyset,(\gamma,\gamma))$ and $u^-$ in $\CM_{v^-,2}((\gamma,\gamma),\emptyset)$ are either cylinders with two negative or positive punctures or pairs of two simple 
holomorphic planes. Since in the latter curves do not carry branch points, it follows that the curves in $\CM^1_{v^+,2,1}(\emptyset,(\gamma,\gamma))$ and 
$\CM^1_{v^-,2,1}((\gamma,\gamma),\emptyset)$ are cylinders, so that the corresponding moduli spaces $\CM^1_{v^-,2}((\gamma,\gamma),\emptyset)$ and 
$\CM_{v^+,2}(\emptyset,(\gamma,\gamma))$ must consist of pairs of simple holomorphic planes. While the latter are automatically regular, it follows from the index and dimension 
additivity that there is an obstruction bundle of rank two over the two-dimensional moduli spaces $\CM^1_{v^+,2,1}(\emptyset,(\gamma,\gamma))$ and 
$\CM^1_{v^-,2,1}((\gamma,\gamma),\emptyset)$, where the two (real) dimensions again account for the second branch point. Since the count of regular curves is clear, it follows 
that the algebraic count of these moduli spaces is given by $\#(\CM^1_{v^+,2,1})^{\bar{\nu}}(\emptyset,(\gamma,\gamma))$ and 
$\#(\CM^1_{v^-,2,1})^{\bar{\nu}}((\gamma,\gamma),\emptyset)$, respectively. \\

In the case when $\Gamma=\gamma^2$, it follows that the curves $u^+$ in $\CM_{v^+,2}(\emptyset,\gamma^2)$ and $u^-$ in $\CM_{v^-,2}(\gamma^2,\emptyset)$ are holomorphic planes. When 
$\CZ(\gamma^2)-2\CZ(\gamma)=-1$ then $\ind(u^+)=2$ and $\ind(u^-)=0$; similarly, when $\CZ(\gamma^2)-2\CZ(\gamma)=+1$ then $\ind(u^+)=0$ and $\ind(u^-)=2$. Since 
$\dim\CM_{v^+,2}(\emptyset,\gamma^2)=\dim\CM_{v^-,2}(\gamma^2,\emptyset)=2$ and $\dim\CM^1_{v^+,2,1}(\emptyset,\gamma^2)=\dim\CM^1_{v^-,2,1}(\gamma^2,\emptyset)=0$, it follows 
from dimension reasons that in the first case we only get contributions from moduli spaces $\CM^1_{v^+,2,1}(\emptyset,\gamma^2)\times\CM_{v^-,2}(\gamma^2,\emptyset)$, while in the 
second case we only get contributions from moduli spaces $\CM_{v^+,2}(\emptyset,\gamma^2)\times\CM^1_{v^-,2,1}(\gamma^2,\emptyset)$. While the curves in 
$\CM^1_{v^+,2,1}(\emptyset,\gamma^2)$ and $\CM^1_{v^-,2,1}(\gamma^2,\emptyset)$ are automatically regular, it follows that we have obstruction bundles of rank two over 
the moduli spaces $\CM_{v^+,2}(\emptyset,\gamma^2)$ and $\CM_{v^-,2}(\gamma^2,\emptyset)$, so that the algebraic count of the moduli spaces is given by 
$\#(\CM_{v^+,2})^{\bar{\nu}}(\emptyset,\gamma^2)$ or $\#(\CM_{v^-,2})^{\bar{\nu}}(\gamma^2,\emptyset)$, respectively. \end{proof}

\subsection{Exceptional spheres cannot break along hyperbolic orbits}
In this final subsection we want to show how the general idea of local symplectic field theory can be used to prove new results about contact hypersurfaces in symplectic four-manifolds. We emphasize that for the proof of the following statement we do \emph{not} need to exclude odd hyperbolic a priori, since for our result we will only need to study multiple covers of the exceptional sphere and over cylinders over orbits appearing in the neck stretching process. 

\begin{theorem} Assume that the exceptional sphere splits after neck-stretching along the unit cotangent bundle of an oriented Lagrangian into punctured holomorphic curves connected by a collection of closed Reeb orbits $\Gamma$ in $V$ which are all Morse nondegenerate. Then at least one of the orbits in $\Gamma$ must be elliptic. \end{theorem}

\begin{proof} After neck-stretching along the contact hypersurface, it is shown in \cite{BEHWZ} that the exceptional sphere breaks along a finite collection $\Gamma$ of closed Reeb orbits on the contact hypersurface. We now want to consider the case where the special point is chosen to lie on the intersection locus $C=\Sigma\cap V=\mp\del\Sigma^{\pm}$ of the exceptional sphere with the stable hypersurface. In this case it follows that, after neck-stretching, the special marked point now lies on one of the orbit cylinders $\IR\times\gamma$ in $\IR\times V$ for $\gamma\in\Gamma$. The corresponding moduli spaces $\CM_{\gamma,2,1}^1(\Gamma^+,\Gamma^-)$ of multiple covers appearing in the boundary of $\CM_{v,2,1}^1$ consists of branched covers of the orbit cylinder with one additional marked point which is required to be a branch point and mapped the special point $(0,0)$ on $\RS\cong\IR\times\gamma$. \\

Now assume that $\gamma$ is hyperbolic. Using the additivity of the Conley-Zehnder index for hyperbolic orbits, it follows that the Fredholm index of a curve in $\CM_{\gamma,d,1}^1(\Gamma^+,\Gamma^-)$ is determined by the Euler characteristic of the underlying punctured curve. While the virtual dimension 
as expected by the Fredholm index continues to be zero for $\CM_{\gamma,2,1}^1((\gamma,\gamma),(\gamma,\gamma))$, it is strictly negative for the moduli spaces 
$\CM_{\gamma,2,1}^1((\gamma,\gamma),\gamma^2)$, $\CM_{\gamma,2,1}^1(\gamma^2,(\gamma,\gamma))$ and $\CM_{\gamma,2,1}^1(\gamma^2,\gamma^2)$. It follows that, after choosing sections in the obstruction bundles over the latter moduli spaces of orbit curves to perturb the Cauchy-Riemann operator, we only need to care about the case when 
$\Gamma^+=\Gamma^-=(\gamma,\gamma)$. \\

Since now both branch points sit over the orbit cylinder of $\gamma$, it follows that we have an obstruction bundle of rank two over the two-dimensional moduli space $\CM_{\gamma,2,1}^1((\gamma,\gamma),(\gamma,\gamma))$, while there is no obstruction bundle for the multiple covers of the other components. While the latter already ensures that we do not need to exclude odd hyperbolic orbits apriori, we now even show that, after perturbing the multiple covers of the orbit cylinder, the latter multi-floor curves cannot contribute. For this we use that the corresponding generating function $\Ih^1_{\gamma,1}$, counting perturbed holomorphic curves in the moduli spaces $(\CM_{\gamma,d,1}^1)^{\bar{\nu}}(\Gamma^+,\Gamma^-)$, was computed in \cite{F3}. Since, for any choice of cokernel bundle sections making the moduli spaces regular, we have $\Ih^1_{\gamma,1}=0$ if $\gamma$ is hyperbolic, it follows that there cannot be a nonzero count of perturbed holomorphic spheres in $\CM_{v,2,1}^1$ when all breaking orbits in $\Gamma$ are hyperbolic. Since the latter contradicts our direct computation from above, we find that at least one orbit in $\Gamma$ must be elliptic. \end{proof}

As an immediate corollary, we find an independent proof of the following result from (\cite{Wel}, theorem 1.3). 

\begin{corollary} Assume that a closed oriented Lagrangian surface $L$ in a closed symplectic four-manifold has a homologically nontrivial intersection with an exceptional sphere $\Sigma$. Then $L$ must be diffeomorphic to $S^2$ or $S^1\times S^1$. \end{corollary}

\begin{proof} Here it suffices to observe that any surface of genus greater than one admits a metric (the uniformizing one) where all closed geodesics are hyperbolic and Morse. Since the same holds true for the corresponding closed Reeb orbits on its unit cotangent bundle, it directly follows from our result that the unit cotangent bundle (for the uniformizing metric) and hence the Lagrangian itself (with the properties stated above) cannot exist. \end{proof}

\end{document}